\documentclass[10pt]{amsart}

\usepackage{amsmath,amsfonts, latexsym,graphicx, footmisc, amssymb, mathtools, enumerate, mdwlist}

\usepackage{hyperref}
\hypersetup{colorlinks=true, urlcolor=blue, citecolor=blue, linkcolor=blue}

\newcommand{\U}{{\mathcal U}}
\newcommand{\0}{{\mathbf 0}}
\newcommand{\C}{{\mathbb C}}
\newcommand{\Z}{{\mathbb Z}}

\newcommand{\Q}{{\mathbb Q}}
\newcommand{\N}{{\mathbb N}}
\newcommand{\cL}{{\mathbb L}}

\newcommand{\strat}{{\mathfrak S}}

\newcommand{\hyp}{{\mathbb H}}

\newcommand{\Fdot}{\mathbf F^\bullet}   

\newcommand{\Pdot}{\mathbf P^\bullet}

\newcommand{\vdual}{{\mathcal D}}

\newtheorem{defn0}{Definition}[section]
\newtheorem{prop0}[defn0]{Proposition}
\newtheorem{conj0}[defn0]{Conjecture}
\newtheorem{thm0}[defn0]{Theorem}
\newtheorem{lem0}[defn0]{Lemma}
\newtheorem{corollary0}[defn0]{Corollary}
\newtheorem{example0}[defn0]{Example}
\newtheorem{remark0}[defn0]{Remark}
\newtheorem{question0}[defn0]{Question}
\newtheorem{exercise0}[defn0]{Exercise}

\newenvironment{defn}{\begin{defn0}}{\end{defn0}}

\newenvironment{thm}{\begin{thm0}}{\end{thm0}}

\newenvironment{exm}{\begin{example0}\rm}{\end{example0}}
\newenvironment{rem}{\begin{remark0}\rm}{\end{remark0}}

\newcommand{\defref}[1]{Definition~\ref{#1}}

\newcommand{\thmref}[1]{Theorem~\ref{#1}}

\newcommand{\mbf}[1]{{\mathbf #1}}

\title{A Note on Numerical Perverse Sheaves}


\author{David B. Massey}

\date{}

\begin{document}

\begin{abstract} We define and describe the properties of a class of perverse sheaves which is very useful when the base ring is not a field.
\end{abstract}

\maketitle




\section{Introduction}

Fix a base ring $R$ that is a regular, Noetherian ring with finite Krull dimension (e.g., $\Z$,
$\Q$, or  $\C$). Let $X$ be a reduced complex analytic space.

It is common when dealing with perverse sheaves of $R$-modules on $X$ to assume that $R$ is a field. One main reason for this, perhaps {\bf the} main reason, is that, with a base field, the support and cosupport conditions for a perverse sheaf are Verdier dual to each other, and so the Verdier dual of a perverse sheaf is again a perverse sheaf. This property is false for general rings $R$. 

For instance, in the case where $X$ consists of a single point and $R=\Z$, a perverse sheaf is simply a complex which is quasi-isomorphic to a finitely-generated $\Z$-module in degree 0 and zeroes in other degrees; however if the finitely-generated $\Z$-module in degree 0 has a non-zero free part and a non-zero torsion part, then the Verdier dual complex will have non-zero cohomology in both degrees $0$ and $1$, and so is not perverse even up to a shift.

However, for many topological applications, we would like to use integral cohomology and so would like to choose our base ring to be $\Z$, which is not a field, but {\bf is} a Dedekind domain. So the question is: is there a ``nice'' class of perverse sheaves over a Dedekind domain  (we include fields as Dedekind domains) such that Verdier dualizing yields perverse sheaves and which behaves well with respect to other important functors?

\medskip

The answer is ``yes'', and we call such complexes of sheaves $\Fdot$ {\it numerical}, which in particular implies that $\Fdot$ is perverse (so, technically, in the phrase ``numerical perverse sheaf'' the word ``perverse'' is superfluous, but we included it in the title for the sake of making our context clear).

We defined and proved properties of numerical sheaves some time ago, in \cite{calcchar}, but the definition and results were in the midst of much other work, and so did not stand out as an independent topic. Hence, we are writing this note to put numerical sheaves in the spotlight.

\bigskip

\section{Definition}

We continue with a base ring $R$ that is a regular, Noetherian ring with finite Krull dimension; we will not need to require that $R$ is a Dedekind domain until we discuss Verdier dualizing numerical sheaves. The definitions and results here are closely related to \cite{MaximSchur}.

Let $\strat$ be a complex analytic Whitney stratification of $X$, with connected strata. Let $\Fdot$ be a bounded complex of sheaves of $R$-modules on $X$, which is constructible with respect to $\strat$. For each $S\in\strat$, we let $d_S:=\dim S$, and let $(\N_S, \cL_S)$ denote {\it complex Morse data for $S$ in $X$}, consisting of a normal slice and complex link of $S$ in $X$; see, for instance, \cite{stratmorse} or \cite{numcontrol}.

\medskip

\begin{defn} For each $S\in\strat$ and each integer $k$, the isomorphism-type of the finitely-generated module
$m_S^k(\Fdot):=\hyp^{k-d_S}(\N_S, \cL_S; \Fdot)$ is independent of the choice of $(\N_S, \cL_S)$; we
refer to
$m_S^k(\Fdot)$ as the {\it degree $k$ Morse module of $S$ with respect to
$\Fdot$}.
\end{defn}

\smallskip

\begin{rem}\label{rem:morsemod} The shift by $d_S$ above is present so that perverse sheaves can have non-zero Morse modules in only degree $0$. 
\end{rem}

\begin{defn}\label{def:numerical} A complex of sheaves $\Pdot$ is {\bf numerical} if and only if the only non-zero Morse modules of strata are in degree $0$ and in degree $0$ are free $R$-modules. \end{defn}

\smallskip

\noindent We usually say just that $\Pdot$ is a numerical sheaf or numerical complex.

\medskip

\begin{rem} \defref{def:numerical} is independent of the stratification, as refining a stratification with respect to which $\Pdot$ is constructible will simply lead to having extra strata with Morse modules that are zero in all degrees. 

We use the term ``numerical'' since the cohomology of the complex Morse data of a given stratum is completely determined by the single number which is the rank in degree $0$. We cannot use the perhaps more-obvious term ``free perverse sheaf'' as ``free'' has a categorical meaning.
\end{rem}

\medskip

The cohomology of the complex Morse data being concentrated in degree zero is equivalent to the complex being {\bf pure with shift $0$} (see Definition 7.5.4 of \cite{kashsch}). This is equivalent to the complex being perverse (\cite{kashsch}, 9.5.2). So we have:

\begin{thm} \textnormal{(Kashiwara and Schapira)} Numerical sheaves are perverse.
\end{thm}

\bigskip

\section{Properties}

As we wrote in the introduction, we proved all of our results on numerical sheaves in \cite{calcchar}, but they were obscured by many other results in that paper.

\medskip

\begin{thm}\label{thm:numprops} Let $f:X\rightarrow\C$ be complex analytic, and let $i:X-V(f)\hookrightarrow X$ be the (open) inclusion. Suppose that $\Pdot$ is a numerical sheaf on $X$. Then, the following sheaves are also numerical:
\begin{enumerate}
\item the nearby cycles, $\psi_f[-1]\Pdot$,
\medskip
\item  the vanishing cycles, $\phi_f[-1]\Pdot$,
\medskip
\item $i_*i^*\Pdot$, 
\medskip
\item $i_!i^!\Pdot$, and
\medskip
\item  the Verdier dual $\vdual\Pdot$, provided that $R$ is a Dedekind domain.
\end{enumerate}
\end{thm}
\begin{proof} Item 1 is Corollary 6.4 of \cite{calcchar}. Item 2 is Corollary 8.4 of \cite{calcchar}. Items 3 and 4 are Corollary 7.8 of \cite{calcchar}.

Item 5 is immediate from Item 7 of Proposition 2.10 of \cite{calcchar}.
\end{proof}

\bigskip

\section{Examples}

In this section, we give some examples, all of which are related to topological data. Hence, {\bf we fix $\Z$ as the base ring}.

\medskip

Suppose that $X$ is a local complete intersection of dimension $d$. In the literature on perverse sheaves, one frequently sees that L\^e proved that $\Z_X^\bullet[d]$ is a perverse sheaf in  \cite{levan}. In fact, what he proved in \cite{levan} is stronger; he proved that each of the complex links of each of the strata of any Whitney stratification of $X$ have the homotopy-type of a bouquet of spheres in middle dimension. In cohomological terms, this means that L\^e proved:

\begin{thm}\label{thm:lenum} \textnormal{(L\^e)} If $X$ is a local complete intersection of pure dimension $d$, then $\Z_X^\bullet[d]$ is numerical.
\end{thm}

\medskip

\begin{exm} Suppose that $X$ is a local complete intersection of pure dimension $d$ embedded in an open subset $\U$ of $\C^{n+1}$ and suppose that $\0\in X$. Let $f:(X, \0)\rightarrow(\C, 0)$ be complex analytic, and let $i:X-V(f)\hookrightarrow X$ be the inclusion. 

Then $\psi_f[-1]\Z^\bullet_X[d]$, $\phi_f[-1]\Z^\bullet_X[d]$, and $i_!i^!\Z^\bullet_X[d]$ are numerical by \thmref{thm:lenum} and \thmref{thm:numprops}.

Now suppose that $z_0$ is the restriction to $V(f)$ of a generic non-zero linear form on $\C^{n+1}$. Then $\Pdot:=\phi_{z_0}[-1]i_!i^!\Z^\bullet_X[d]$ is numerical and either $\0$ is not in the support  of $\Pdot$ or $\0$ is an isolated point in the support of $\Pdot$. Therefore the stalk cohomology $H^k(\Pdot)_\0$ is zero if $k\neq 0$ and $H^0(\Pdot)_\0$ is free abelian. This is of topological interest since $H^0(\Pdot)_\0\cong H^{d-1}(\cL_{X,\0}, \cL_{V(f), \0}; \,\Z)$, where $\cL$ denotes the complex link of a space at a point.

\end{exm}

\medskip

\begin{exm} Suppose that $X$ is embedded in an open subset $\U$ of $\C^{n+1}$ and suppose that $\0\in X$. Let $\Pdot$ be a numerical sheaf on $X$. Let $\mbf z:=(z_0, z_1, \dots, z_n)$ be a generic linear choice of coordinates for $\C^{n+1}$. We will not distinguish in our notation between the $z_i$'s and their restrictions to various subsets.

In Theorem 5.4  of \cite{numinvar} (and also in Theorem 10.9 of \cite{lecycles}), we proved that there is a chain complex of $\Z$-modules with the degree  $-i$ term, for $0\leq i\leq\dim_0 X$, equal to
$$
H^0(\phi_{z_i}[-1]\psi_{z_{i-1}}[-1]\dots\psi_{z_0}[-1]\Pdot)_\0
$$ 
such that the cohomology in any degree $k$ is isomorphic to $H^k(\Pdot)_\0$. We called the ranks of these $\Z$-modules the {\it characteristic polar multiplicities}, and denoted the rank of the degree $-i$ cohomology by $\lambda^i_{\Pdot, \mbf z}$.

Now, since $\Pdot$ is numerical, we now know that each $H^0(\phi_{z_i}[-1]\psi_{z_{i-1}}[-1]\dots\psi_{z_0}[-1]\Pdot)_\0$ is free abelian, and so we have a chain complex
$$
0\rightarrow\Z^{\lambda^n_{\Pdot, \mbf z}}\rightarrow\Z^{\lambda^{n-1}_{\Pdot, \mbf z}}\rightarrow\cdots\rightarrow \Z^{\lambda^{1}_{\Pdot, \mbf z}}\rightarrow \Z^{\lambda^{0}_{\Pdot, \mbf z}}\rightarrow 0
$$
whose homology/cohomology at the $\Z^{\lambda^i_{\Pdot, \mbf z}}$ term is isomorphic to $H^{-i}(\Pdot)_\0$.

\medskip

In Corollary 10.10 of \cite{lecycles}, we applied this to the case where $\Pdot=\phi_f[-1]\Z^\bullet_\U[n+1]$ for \hbox{$f:\U\rightarrow\C$} a complex analytic function and concluded, in this case, that $\lambda^i_{\Pdot, \mbf z}$ was equal to the $i$-th L\^e number of $f$ with respect to $\mbf z$ at the origin, where the L\^e numbers $\lambda^i_{f, \mbf z}$ are actually defined algebraically. However, we had to switch to $\C$ as the base ring in order to give our chain complex because we did not have the results on numerical sheaves. Now we know that in the chain complex in Corollary 10.10 of \cite{lecycles}, we can replace the $\C$'s by $\Z$'s (and we also need to correct the typographical error which has the indexing reversed on the L\^e numbers).

\end{exm}

\bigskip

\bibliographystyle{plain}

\bibliography{Masseybib}

\end{document}